\documentclass[12pt]{article}
\usepackage{a4}

\usepackage{amsthm,amssymb,amsmath}
\usepackage{enumitem}
\usepackage{tikz}

\newtheorem{theorem}{Theorem}

\newtheorem{lemma}[theorem]{Lemma}
\newtheorem{observation}[theorem]{Observation}
\newtheorem{corollary}[theorem]{Corollary}
\newtheorem{conjecture}[theorem]{Conjecture}

\newcommand\size[1] {\left|{#1}\right|}
\newcommand\Set[2] {\left\{{#1}:\,{#2}\right\}}
\newcommand\Setx[1] {\left\{{#1}\right\}}

\newcommand{\sm}{\setminus}

\newcommand{\proj}[1]{\mathbb{P}^{#1}}
\newcommand{\white}[1]{#1^\circ}
\newcommand{\black}[1]{#1^\bullet}

\newcommand{\IB}{\mathsf{IB}}
\newcommand{\EB}{\mathsf{EB}}
\newcommand{\RR}{\mathbb R}
\newcommand{\CC}{\mathsf C}
\newcommand{\KK}{\mathsf K}
\newcommand{\LL}{\mathsf L}

\newcommand{\lex}{\prec_{\mathrm{lex}}}
\newcommand{\subs}[2]{V(#1,#2)}
\newcommand{\qk}[2]{\mathsf{QK}(#1,#2)}

\newcommand{\qbblack}[2]{\mathsf{QB}^{\bullet}(#1,#2)}
\newcommand{\qbwhite}[2]{\mathsf{QB}^{\circ}(#1,#2)}
\newcommand{\qbblackx}[2]{\mathsf{QB}^{\bullet}_1(#1,#2)}
\newcommand{\qbwhitex}[2]{\mathsf{QB}^{\circ}_1(#1,#2)}
\newcommand{\level}[2]{\ell^{#1}(#2)}
\newcommand{\layer}[3]{V_{#1}(#2,#3)}
\newcommand{\V}[3][\,]{V_{#1}(#2,#3)}

\newcommand{\core}[1]{\mathrm{core}(#1)}

\newcommand{\clone}[2]{#2\langle#1\rangle}
\newcommand{\fig}[1]{\raisebox{-.5\height}{%
    \includegraphics[page=#1]{QG-figs}}}
\newcommand{\hf}{\hspace*{0mm}\hfill}

\title{\textbf{Schrijver graphs\\and projective quadrangulations}}%
\author{Tom\'{a}\v{s} Kaiser\thanks{Department of Mathematics,
    Institute for Theoretical Computer Science (CE-ITI), and European
    Centre of Excellence NTIS (New Technologies for the Information
    Society), University of West Bohemia, Pilsen, Czech
    Republic. Supported by project GA14-19503S of the Czech Science
    Foundation. Email: \texttt{kaisert@kma.zcu.cz}.}
  \and
    Mat\v{e}j Stehl\'{\i}k\thanks{Laboratoire G-SCOP, Universit\'e
    Grenoble Alpes, France. Partially supported by ANR project Stint
    (ANR-13-BS02-0007) and by LabEx PERSYVAL-Lab (ANR-11-LABX-0025).
    Email: \texttt{matej.stehlik@grenoble-inp.fr}.}
  }
\date{}

\begin{document}
\maketitle

\begin{abstract}
  In a recent paper [\textit{J.\ Combin.\ Theory Ser.\ B}, 113 (2015), pp.~1--17],
  the authors have extended the concept
  of quadrangulation of a surface to higher dimension, and showed that
  every quadrangulation of the $n$-dimensional projective space
  $\proj{n}$ is at least $(n+2)$-chromatic, unless it is bipartite.
  They conjectured that for any integers $k\geq 1$ and $n\geq 2k+1$,
  the Schrijver graph $SG(n,k)$ contains a spanning subgraph which is
  a quadrangulation of $\proj{n-2k}$. The purpose of this paper is to
  prove the conjecture.
\end{abstract}


\section{Introduction}
\label{sec:intro}

Given any integers $k\geq 1$ and $n\geq 2k$, the \emph{Kneser graph}
$KG(n,k)$ is the graph whose vertex set consists of all $k$-subsets of
$[n]=\Setx{1,\dots,n}$, and with edges joining pairs of disjoint
subsets. It was conjectured by Kneser~\cite{Kne55}, and proved by
Lov\'{a}sz~\cite{Lov78} in 1978, that the chromatic number of $KG(n,k)$
is $n-2k+2$.

Schrijver~\cite{Sch78} found a vertex-critical subgraph $SG(n,k)$ of
$KG(n,k)$ whose chromatic number is also $n-2k+2$. (Recall that a
graph is \emph{vertex-critical} if the deletion of any vertex
decreases the chromatic number.)

In~\cite{KS15}, a \emph{quadrangulation} of a space triangulated by a
(generalised) simplicial complex $\KK$ is defined as a spanning
subgraph $G$ of the $1$-skeleton $\KK^{(1)}$ such that the induced
subgraph of $G$ on the vertex set of any maximal simplex of $\KK$ is
complete bipartite with at least one edge. 

Particular attention was given in~\cite{KS15} to quadrangulations of
projective spaces, and it was shown that if $G$ is a quadrangulation
of the projective space $\proj n$, then the chromatic number of $G$ is
at least $n+2$. By constructing suitable projective quadrangulations
of $\proj n$ homomorphic to Schrijver graphs, an alternative proof of
Schrijver's result was obtained.

The purpose of this paper is to prove Conjecture 7.1 from~\cite{KS15}
by establishing the following result:
 
\begin{theorem}\label{t:main}
  For any $k\geq 1$ and $n > 2k$, the graph $SG(n,k)$ contains a
  spanning subgraph $QG(n,k)$ that embeds in $\proj{n-2k}$ as a
  quadrangulation. In particular, $\chi(QG(n,k)) = n-2k+2$.
\end{theorem}

To prove Theorem~\ref{t:main}, we need to construct a suitable
triangulation of the sphere $S^{n-2k}$. We first review some
topological preliminaries (Section~\ref{sec:top}) and explore
combinatorial relations among the vertices of Schrijver graphs
(Section~\ref{sec:combin}).

In Section~\ref{sec:construction}, the required properties of the
sought triangulation of $S^{n-2k}$ are formulated in
Theorem~\ref{t:emb}, which is then proved by giving an explicit
recursive construction. Theorem~\ref{t:main} is derived at the end of
Section~\ref{sec:construction}.

In Section~\ref{sec:properties}, two open problems are given to
conclude the paper. In particular, we conjecture that the graph
$QG(n,k)$ of Theorem~\ref{t:main} is edge-critical.


\section{Topological preliminaries}
\label{sec:top}

In this section, we recall the necessary topological concepts. For a
background on topological methods in combinatorics, we refer the
reader to Matou\v{s}ek~\cite{Mat03}. For an introduction to algebraic
topology, consult Hatcher~\cite{Hat02} or Munkres~\cite{Mun84}.

A \emph{simplicial complex} $\CC$ with vertex set $V$ is a hereditary
set system on $V$; the elements of this set system are the
\emph{faces} of $\CC$. A \emph{geometric simplicial complex} $\KK$ in
$\RR^d$ is obtained if we associate each vertex in $V$ with a point in
$\RR^d$ in such a way that
\begin{enumerate}[label=(\arabic*)]
\item the set of points $P_\sigma$ associated with each face $\sigma$
  is in convex position, and
\item for distinct faces $\sigma$ and $\tau$, the relative interiors
  of the convex hulls of $P_\sigma$ and $P_\tau$ are disjoint.
\end{enumerate}
The convex hulls of the sets $P_\sigma$, where $\sigma$ is a face of
the underlying simplicial complex $\CC$, will be referred to as the
\emph{faces} of $\KK$. Since we will be dealing exclusively with
geometric simplicial complexes in this paper, we will often drop the
adjectives `geometric' and `simplicial'.

A face such as $\Setx{a,b,c}$ is also written as $abc$. Two vertices
$v,w$ of $\KK$ are \emph{adjacent} if $vw$ is a face of $\KK$. The
\emph{dimension} of a face $\sigma$ is $\size{\sigma}-1$. Faces of
dimension one are called \emph{edges}. The vertex set of a geometric
simplicial complex $\KK$ will be referred to as $V(\KK)$.

The \emph{space} $\|\KK\|$ of a geometric simplicial complex $\KK$ in
$\RR^d$ is the subspace of $\RR^d$ obtained as the union of all faces
of $\KK$. If a space $X\subseteq\RR^d$ is homeomorphic to $\|\KK\|$,
we say that $\KK$ \emph{triangulates} $X$.

The \emph{induced subcomplex} of $\KK$ on a set $X\subseteq V(\KK)$,
denoted by $\KK[X]$, has vertex set $X$ and its faces are all the
faces of $\KK$ contained in $X$.

A \emph{$2$-coloured complex} $\KK$ in $\RR^d$ is a geometric
simplicial complex in $\RR^d$, with each vertex coloured black or
white. For any point $p\in\RR^d$, its \emph{antipode} is the point
$-p$. The complex $\KK$ is \emph{antisymmetric} if the antipode $-v$
of every vertex $v$ is also a vertex of $\KK$, and the colours of $v$
and $-v$ are different. 

Suppose that $\KK$ triangulates the ball $B^d$. The \emph{boundary} of
$\KK$ is the subcomplex triangulating the boundary sphere $S^{d-1}
= \partial B^d$. We will say that $\KK$ is
\emph{boundary-antisymmetric} if its boundary is antisymmetric.

Let us recall the definition of deformation retraction (as given
in~\cite{Hat02}). Given a subspace $A$ of a topological space $X$, a
family of continuous maps $f_t:\,X\to X$ (where $t\in[0,1]$) is a
\emph{deformation retraction} of $X$ onto $A$ if $f_0$ is the
identity, so is the restriction of each $f_t$ to $A$, the image of
$f_1$ is $A$, and the family is continuous when viewed as a map from
$X\times[0,1] \to X$. If such a deformation retraction exists, $A$ is
said to be a \emph{deformation retract} of $X$.

Next, let $\KK$ be a 2-coloured geometric complex whose space is a
deformation retract of the thickened sphere $S^d\times I$ in
$\RR^{d+1}$, where $d\geq 1$, $S^d$ is the unit $d$-sphere and $I$ is
a short interval in $\RR$. Thus, we can define the \emph{interior} of
$\KK$ as the bounded component of $\RR^{d+1}\sm\|\KK\|$, and similarly
for the \emph{exterior} of $\KK$. Note that the origin of $\RR^{d+1}$
is contained in the interior. We define the \emph{interior boundary}
of $\KK$, $\IB(\KK)$, as the subcomplex of $\KK$ induced on the set of
vertices contained in the closure of the interior of $\KK$. The
\emph{exterior boundary} $\EB(\KK)$ is defined analogously. Note that
$\IB(\KK)$ and $\EB(\KK)$ need not be disjoint.

In the above setting, we will utilise the operation of adding the
\emph{clone} of a vertex. For a vertex $v$ of $\IB(\KK)$, we add a
vertex $v^*$ of the same colour (the \emph{clone} of $v$) and embed it
in the open segment from $v$ to the origin, very close to
$v$. Furthermore, for each face $\sigma$ of $\IB(\KK)$ containing $v$,
we add a face $\sigma\cup\Setx{v^*}$. Thus, $v^*$ replaces $v$ in the
interior boundary of the resulting complex $\KK'$. Note also that
$vv^*$ is a face of $\KK'$.

Let $K$ be a $2$-coloured complex and $u,v$ two adjacent vertices of
$\KK$ of the same colour. The \emph{contraction} of the edge $uv$ is
the operation replacing each face $\sigma$ of $\KK$ with
$\sigma\sm\Setx{u,v}\cup\Setx w$, where $w$ is a new vertex (assigned
the colour of $u$ and $v$). Geometrically, it corresponds to shrinking
the segment $uv$ to a point. By definition, the operation does not
introduce multiple copies of any face. For example, if $\KK$ is the
complex whose maximal faces are $xu$, $xv$ and $uvy$, where $u$ and
$v$ are black and $x$ and $y$ are white, then the contraction of $uv$
produces the complex with maximal faces $xw$ and $wy$.

Let $\KK$ and $\LL$ be $2$-coloured complexes. A mapping $f:\,V(\KK)\to
V(\LL)$ is a \emph{homomorphism} (of $2$-coloured complexes) from
$\KK$ to $\LL$ if $f$ preserves vertex colours and for any face
$\sigma$ of $\KK$, its image $f[\sigma]$ is a face of $\LL$. (We
stress that $f[\sigma]$ is a set, without repeated elements.)

A homomorphism $f$ from $\KK$ to $\LL$ is an \emph{isomorphism} if $f$
is a bijection and $f^{-1}$ is a homomorphism.

For an antisymmetric $2$-coloured complex $\KK$ triangulating a
sphere, we define its \emph{associated graph} $G(\KK)$ as the graph
with vertex set $V(\KK)$ and with the edge set consisting of all
edges of $\KK$ with one end black and the other white.


\section{Combinatorial preliminaries}
\label{sec:combin}

Before we present the construction proving Theorem~\ref{t:main}, we
need to do some preparatory work. In this section, we introduce some
terminology and notation that is useful for the classification of the
vertices of the Schrijver graph $SG(n,k)$. 

Let $k\geq 1$ and $n\geq 2k+1$. We let $C_n$ be the $n$-circuit on the
vertex set $[n] = \Setx{1,\dots,n}$ and let $\subs n k$ be the set of
all independent subsets of $[n]$ of size $k$. Addition and subtraction
on $[n]$ are defined `with wrap-around': for instance, if $i,j\in [n]$
and $(i-1) + (j-1) \equiv \ell-1\pmod n$, where $\ell\in[n]$, then
$i+j$ is defined as $\ell$. We let $\V n k$ be the set of all subsets
of $[n]$ of size $k$ that are independent sets in $C_n$. Note that
$\V{n-1}k$ is a subset of $\V n k$.

The \emph{core} of a set $A\in\V n k$ is the set
\begin{equation*}
  \core A =
  \begin{cases}
    A\sm\Setx 1 & \text{if $1\in A$,}\\
    A\sm\Setx{\max(A)} & \text{otherwise.}
  \end{cases}
\end{equation*}
Thus, $\core{\Setx{1,3,5}} = \Setx{3,5}$, while $\core{\Setx{2,4,6}} =
\Setx{2,4}$.

\begin{observation}\label{obs:extreme}
  For $A\in\V n k$,
  \begin{equation*}
    \core A\cap\Setx{1,n} = \emptyset.
  \end{equation*}
\end{observation}

Let $0 \leq i \leq n/2$. We define the set $\Lambda_{n,i}\subseteq [n]$
as follows:
\begin{equation*}
  \Lambda_{n,i} =
  \begin{cases}
    \Setx{2,4,\dots,i-1}\cup\Setx{n-i+1,n-i+3,\dots,n} & \text{if $i$ is odd,}\\
    \Setx{1,3,\dots,i-1}\cup\Setx{n-i+1,n-i+3,\dots,n-1} & \text{if $i$
      is even.}
  \end{cases}
\end{equation*}

For small $i$, the sets $\Lambda_{n,i}$ are given in the following table:
\begin{center}
  \begin{tabular}{c|c|c|c|c}
    $\Lambda_{n,0}$ & $\Lambda_{n,1}$ & $\Lambda_{n,2}$ & $\Lambda_{n,3}$ & $\Lambda_{n,4}$\\\hline
    $\emptyset$ & $\Setx n$ & $\Setx{1,n-1}$ & $\Setx{2,n-2,n}$ & $\Setx{1,3,n-3,n-1}$
  \end{tabular}
\end{center}
Note that for each $i$, $\Lambda_{n,i}\in\V n i$. 

The $n$-\emph{level} of a set $A\in \V n k$, $\level n A$, is the
maximum $i$ such that $\Lambda_{n,i} \subseteq A$. Note that $0 \leq
\level n A\leq k$. For $0\leq i\leq k$, we define
\begin{equation*}
  \V[i] n k = \Set{A\in \V n k}{\level n A = i}.
\end{equation*}
Furthermore, we let $\V[+] n k$ be the union of all $\V[i] n k$
with $i\geq 1$.

\begin{lemma}\label{l:level}
  We have
  \begin{equation*}
    \V[0] n k = \V{n-1} k.
  \end{equation*}
\end{lemma}
\begin{proof}
  We need to show that for any set $A\in\V n k$, we have $\level n A
  = 0$ if and only if $A\in\V{n-1}k$. By definition, $\level n A = 0$
  if and only if $A$ contains neither $\Setx n$ nor $\Setx{1,n-1}$ as
  a subset. In turn, this holds if and only if $A\in \V{n-1} k$.
\end{proof}

Let $B\in \V{n-1} k$. We define the set $\clone n B \in\V n k$ by
\begin{equation*}
  \clone n B = \core B \cup\Setx{n}.
\end{equation*}
By Observation~\ref{obs:extreme}, the operation is well-defined.
Since it will be used in relation with adding the `clone' of a vertex
labelled by $B$, we might call $\clone n B$ the \emph{$n$-clone} of
$B$.

The following lemma will be useful:
\begin{lemma}\label{l:clone-clone}
  For $2k+1 \leq i < m$ and $A\in\V{i-1}k$, $\clone m {(\clone i A)} =
  \clone m A$.
\end{lemma}
\begin{proof}
  By definition, $\clone i A = \core A\cup\Setx i$. Since
  $1\notin \core A$, $\core{\clone i A} = \core A$. Thus, $\clone m
  {(\clone i A)} = \clone m A$.
\end{proof}

For a set $A\in\V n k$ such that $1\notin A$, we define $A-1$ as the
set obtained by subtracting $1$ from each element of $A$ (and
similarly for $A+1$).

Let us define a mapping $f$ from $\V n k$ to $\V{n-2}{k-1}$, and a
mapping $g_n$ in the inverse direction. Let $X \in \V n k$ and $Y\in
\V{n-2}{k-1}$. The mappings are as follows:
\begin{align*}
  f(X) &= \core X - 1,\\
  g_n(Y) &=
  \begin{cases}
    (Y + 1) \cup \Setx 1 & \text{if $n-2\in Y$,}\\
    (Y + 1) \cup \Setx n & \text{otherwise.}     
  \end{cases}
\end{align*}

\begin{lemma}\label{l:bijection}
  The restriction of $f$ to $\V[+] n k$ is a bijection
  \begin{equation*}
    f:\,\V[+] n k \to \V{n-2}{k-1},
  \end{equation*}
  and $g_n$ is its inverse. Furthermore, $f$ maps disjoint pairs of
  sets to disjoint pairs.
\end{lemma}
\begin{proof}
  The first assertion follows from the fact that the image of $g_n$ is
  contained in $\V[+] n k$, and from the easily verified
  equalities
  \begin{equation*}
    f(g_n(Y)) = Y \text{\qquad and \qquad} g_n(f(X)) = X
  \end{equation*}
  for $X\in\V[+] n k$, $Y\in\V{n-2}{k-1}$.

  The assertion that the images of disjoint sets under $f$ are
  disjoint follows directly from the definition of $f$.
\end{proof}

\begin{corollary}\label{cor:cores}
  All the sets in $\V[+] n k$ have distinct cores.
\end{corollary}

\begin{observation}
  For any set $B\in \V{n-1}k$, we have $f(\clone n B) = f(B)$.
  Thus, the suitable restriction of $f$ is a bijection
  \begin{equation*}
    \Set{\clone n B}{B\in\V[+]{n-1}k} \to \V{n-3}{k-1}.
  \end{equation*}
\end{observation}




To avoid ambiguity in our construction, we will need to fix a suitable
total order on each set $\V n k$. It will be convenient to simply use
the lexicographical ordering: for $A,B\in\V n k$, let $A'$ and $B'$
be the sequences obtained by listing the elements of $A$ and $B$
(respectively) in the increasing manner, and define $A\lex B$ if $A'$
precedes $B'$ in the standard lexicographical ordering. 

Finally, we define a set $A\in\V n k$ to be \emph{singular} if
$A\in\V{2k}k$. Thus, the singular sets in $\V 7 3$ are $\Setx{1,3,5}$
and $\Setx{2,4,6}$.


\section{Constructing the embedding}
\label{sec:construction}

In this section, we shall construct the antisymmetric $2$-coloured
complex $\qk n k$ in $\RR^{n-2k+1}$ triangulating the sphere
$S^{n-2k}$. The vertices will be coloured black and white; both the
black vertices and the white vertices will be labelled bijectively
with elements of $\V n k$. We will identify each vertex with its label
and speak, for instance, of the black copy of $\Setx{1,3,5}$ or the
white copy of $\Setx{2,6,8}$. For a set $A\in\V n k$, its black copy
will be denoted by $\black A$ and its white copy by $\white A$.

\begin{theorem}\label{t:emb}
  For any $k\geq 1$ and $n\geq 2k+1$, there is a 2-coloured geometric
  complex $\qk n k$ in $\RR^{n-2k+1}$ with the following properties:
  \begin{enumerate}[label=(\roman*)]
  \item $\qk n k$ is an antisymmetric triangulation of the sphere
    $S^{n-2k}$ such that no face contains a pair of antipodal
    vertices.
  \item $\qk n k$ contains no monochromatic maximal faces.
  \item The associated graph of $\qk n k$ is a spanning subgraph of
    $SG(n,k)$.
  \item For $n > 2k+1$, $\qk n k$ contains $\qk{n-1}k$ as an
    antisymmetric subcomplex.
  \end{enumerate}
\end{theorem}

Let us embark on the construction of $\qk n k$ which eventually proves
Theorem~\ref{t:emb}. In the construction, we will ensure that the
following (more technical) conditions hold as well:
\begin{enumerate}[label=(P\arabic*)]
\item\label{i:close} If $\black A\white B$ is a face of $\qk n k$,
  then $\size{\level n A - \level n B} \leq 1$, and $\level n A =
  \level n B$ only if $\level n A = \level n B = 0$.
\item\label{i:sing} For $k\geq 2$ and a vertex $\black A$ of $\qk n
  k$, $A$ is nonsingular if and only if for any
  $B'\subseteq\Setx{2,\dots,n-1}$, $\black A$ is contained in a face
  not containing any vertex $\white B$ with $\core B = B'$.
\item\label{i:sing-same} For a vertex $\black A$ with $A$ singular and
  an adjacent vertex $\black B$, $\core A = \core B$.
\end{enumerate}

The definition of $\qk n k$ is straightforward in case that $n =
2k+1$. For $j\in[2k+1]$, let
\begin{equation*}
  I_j = \Setx{j, j+2, \dots, j+2k-2} \in \V{2k+1} k.
\end{equation*}
The complex $\qk{2k+1}k$ is $1$-dimensional, so we can describe it as
a graph: it is the circuit of length $2(2k+1)$ with vertices
\begin{equation*}
  \black{I_1}, \white{I_2}, \black{I_3}, \dots, \black{I_{2k+1}},
  \white{I_1}, \black{I_2}, \dots, \white{I_{2k+1}}, \black{I_1} 
\end{equation*}
in this order. See Figure~\ref{fig:qk73} for an illustration.

\begin{figure}
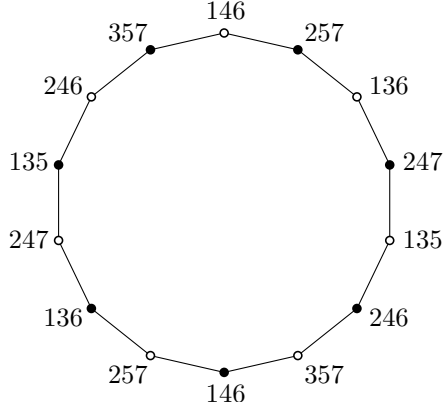

\begin{center}
  \fig1
\end{center}
\caption{The complex $\qk 7 3$. Set brackets are omitted in vertex
  labels such as $\Setx{1,3,5}$.}
\label{fig:qk73}
\end{figure}

Suppose thus that $n > 2k+1$ and that $\qk{n-1}k$ has already been
constructed. Recall that $\qk{n-1}k$ is an antisymmetric triangulation
of $S^{n-2k-1}$ in $\RR^{n-2k}$. A quick summary of the construction
of $\qk n k$ is as follows:
\begin{itemize}
\item we extend $\qk{n-1}k$ to $\qbblackx n k$ (triangulating a
  `thickened sphere' $S^{n-2k-1}\times I$ if $k\geq 2$),
\item we obtain boundary-antisymmetric triangulations $\qbblack n k$
  (and $\qbwhite n k$) of the $(n-2k)$-ball $B^{n-2k}$ by filling in
  the interior of $\qbblackx n k$ using $\qbwhite{n-2}{k-1}$ (in the
  case of $\qbwhite n k$, inverting the colours),
\item we form an antisymmetric triangulation of $S^{n-2k}$ from
  $\qbblack n k$ and $\qbwhite n k$.
\end{itemize}

As the first step of the construction, we extend $\qk{n-1} k$ to a
2-coloured complex $\qbblackx n k$ by adding the `clones' of some of
the vertices, and contracting certain edges.  Both the exterior
boundary $\qk{n-1}k$ of $\qbblackx n k$ and its interior boundary will
be deformation retracts of $\qbblackx n k$. The interior boundary of
$\qbblackx n k$ will be shown to be isomorphic (as a complex) to
$\qk{n-3}{k-1}$, enabling us to fill in the interior by recursion.

There are two special cases where the construction is particularly
simple: $k=1$ and $n=2k+2$. Let us begin with $k=1$. In this case,
$\qbblackx n 1$ is obtained just by taking the cone over $\qk{n-1}1$,
with the newly added apex vertex $\black n$ placed at the origin.

In the case $n=2k+2$, we also construct $\qbblackx n k$ directly from
$\qk{2k+1}k$. For each vertex $\black{I_j}$, where
$j\in[2k+1]\sm\Setx{1,2}$, add its clone $\black{\clone n
  {I_j}}$. Furthermore, add the faces of the following complexes:
\begin{itemize}
\item the join of $\black{\clone n {I_3}}$ with the induced
  subcomplex of $\qk{2k+1}k$ on the set
  $\Setx{\white{I_{2k+1}},\black{I_1},\white{I_2}}$,
\item the join of $\black{\clone n {I_{2k+1}}}$ with the induced
  subcomplex of $\qk{2k+1}k$ on
  $\Setx{\white{I_1},\black{I_2},\white{I_3}}$.
\end{itemize}
See Figure~\ref{fig:qb1base} for an illustration.

\begin{figure}
  \centering
  \fig5\hf\fig2
  \caption{The complexes $\qbblackx41$ (left) and $\qbblackx83$
    (right). Set brackets in vertex labels are omitted. In
    $\qbblackx83$, the clones $\black{(\clone n {I_3})} = \black{358}$
    and $\black{(\clone n {I_7})} = \black{248}$ have been moved
    slightly to produce a more symmetric picture.}
  \label{fig:qb1base}
\end{figure}

To construct $\qbblackx n k$ for $n > 2k+2$ and $k > 1$, we proceed as
follows (the process is illustrated in Figure~\ref{fig:qb1}):
\begin{enumerate}[label=(B\arabic*)]
\item\label{i:black-out} For each vertex $\black A$ with
  $A\in\layer+{n-1}k$, we add its clone $\black{\clone n A}$.
\item\label{i:black-equ} For each vertex $\black A$ with
  $A\in\layer0{n-1}k$, in the order given by $\lex$, we add a
  `temporary' clone denoted by $\black{A_*}$.
\item\label{i:black-contract} For each vertex $\black A$ with
  $A\in\layer0{n-1}k$, we note that the vertex $\black{\clone n
    {(\clone{n-1}A)}}$, added in step~\ref{i:black-out}, is adjacent
  to $\black{A_*}$. We contract the face consisting of these two
  vertices. The resulting vertex retains the label $\black{\clone n
    {(\clone{n-1}A)}}$ (which is the same as $\black{\clone n A}$).
\item\label{i:white} For each vertex $\white B$ with nonsingular
  $B\in\layer0{n-1}k$, in the order given by $\lex$, we add its
  temporary clone $\white{B_*}$. (The fact that $\white B$ is
  contained in the interior boundary follows from
  Lemma~\ref{l:bd}(ii) below.)
\item\label{i:last}%
  \label{i:white-contract}%
  For each vertex $\white B$ with nonsingular $B\in\layer0{n-1}k$, we
  note that $\white{B_*}$ is adjacent to the vertex
  $\white{\clone{n-1}B}$, added in step~\ref{i:white}, and we contract
  the face consisting of these two vertices. The resulting vertex
  retains the label $\white{\clone{n-1}B}$.
\end{enumerate}

\begin{figure}
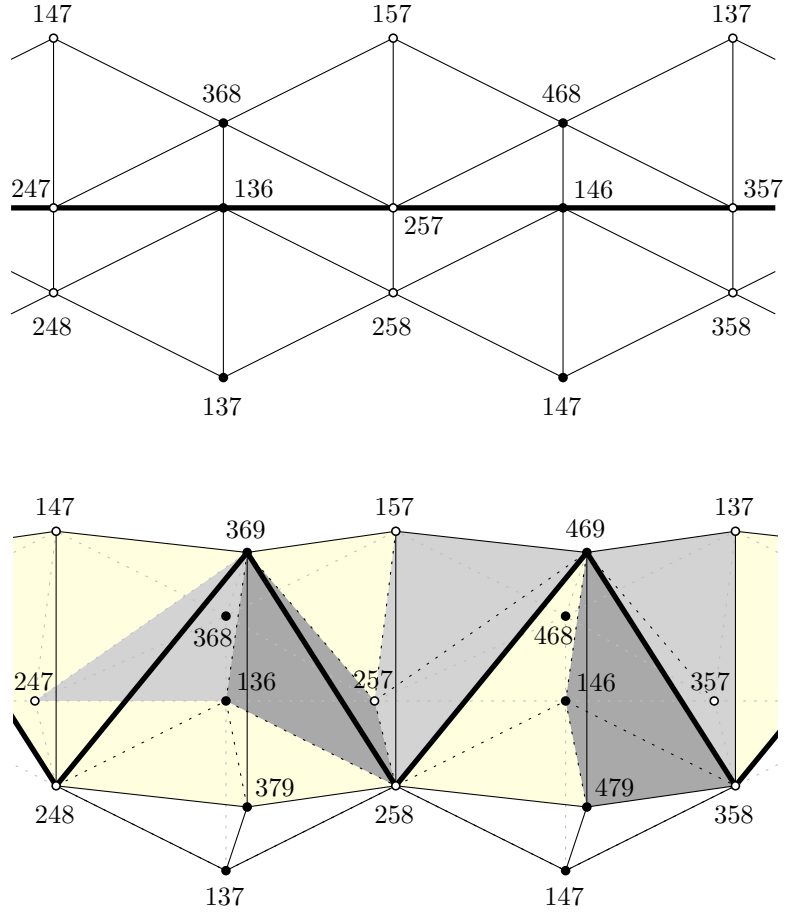

  \centering
  \fig6\\[1cm]
  \fig7
  \caption{The construction of $\qbblackx 9 3$. Above: a portion of
    $\qk 8 3$; the equator $\qk 7 3$ is shown bold. Below: the result
    of steps (B1)--(B6). The original complex $\qk 8 3$ should be
    pictured in a base plane, and the added clones (such as
    $\black{369}$, but not $\white{157}$) above it. Solid and dashed
    lines represent visibility. Some of the $3$-dimensional faces have
    been shaded to make the structure of the complex easier to
    visualise. Faces of the interior boundary of $\qbblack 9 3$ are
    coloured light yellow. The thick line is $\qk 5 2$, the equator of
    $\IB(\qbblackx 9 3)\simeq \qk 6 2$. Notice that it is disjoint
    from the equator of $\qk 8 3$.}
  \label{fig:qb1}
\end{figure}

By switching colours in the above description (for example, adding
clones to white vertices in step \ref{i:black-out}), we obtain the
complex $\qbwhitex n k$.

\begin{lemma}\label{l:bd}
  Let $\KK_{123}$ be the complex resulting from
  steps~\ref{i:black-out}--\ref{i:black-contract} of the above
  construction, and let $A\in\V{n-1}k$. The following properties hold:
  \begin{enumerate}[label=(\roman*)]
  \item $\black A$ is not contained in the interior boundary of
    $\KK_{123}$.
  \item If $A\in\V[0]{n-1}k$ is nonsingular, then $\white B$ is
    contained in the interior boundary of $\KK_{123}$.
  \end{enumerate}
\end{lemma}

\begin{lemma}\label{l:interior}
  Let $\IB$ be the interior boundary of $\qbblackx n k$, where $n\geq
  2k+2$. 
  The following properties hold:
 \begin{enumerate}[label=(\roman*)]
 \item The vertex set of $\IB$ is
   \begin{equation*}
     V_\IB = \bigcup_{A\in \V[+]{n-1}k} \Setx{\black{\clone n A}, 
       \white A}.
   \end{equation*}
 \item $\IB$ is the induced subcomplex of $\qbblackx n k$ on $V_\IB$.
 \item $\IB$ is isomorphic to $\qk{n-3}{k-1}$, with the isomorphism
   determined by the mapping $f:\,A \mapsto \core A-1$ and preserving
   the colours (where $A$ is a vertex of $\IB$).
 \end{enumerate}
\end{lemma}
\begin{proof}
  (i) For $n=2k+2$ or $k=1$, the assertion is easy to verify
  directly. Assume then that $n > 2k+2$ and $k > 1$. In
  steps~\ref{i:black-out} and \ref{i:black-equ}, we added clones of
  all black vertices of $\qk{n-1}k$. Thus, these vertices are not
  included in the interior boundary of the resulting complex, and this
  remains true after the contractions performed in
  step~\ref{i:black-contract}. On the other hand, all the added clones
  are vertices of $\IB$.

  For a similar reason, in view of step \ref{i:white}, no vertex
  $\white A$ with nonsingular $A\in\V[0]{n-1}k$ is a vertex of
  $\IB$. If $\white A$ is such that $A\in\V[0]{n-1}k$ is singular,
  then property~\ref{i:sing} of $\qk{n-1}k$ implies that for all the
  vertices $\black B$ adjacent to $\white A$, the sets $\core B$ are
  the same. This means that the complex resulting from
  steps~\ref{i:black-out}--\ref{i:black-contract} will contain the
  cone over the star of $\white A$, with apex $\black{\clone n B}$,
  where $\black B$ is any black vertex adjacent to $\white A$ in $\qk
  {n-1} k$. Thus, $\white A$ is not a vertex of $\IB$.

  It remains to show that any vertex $\white A$ with $A\in\V[+]{n-1}k$
  is a vertex of $\IB$. In this case, $A$ is nonsingular, so by
  property~\ref{i:sing} of $\qk{n-1}k$, for any vertex $\black B$
  adjacent to $\white A$, $\white A$ is adjacent to a face not
  containing any vertex $\black C$ with $\core B = \core C$. This
  implies that step~\ref{i:black-contract} does not eliminate $\white
  A$ from the interior boundary of the resulting complex.

  (ii) Since $\IB$ is, trivially, a subcomplex of $\qbblackx n k$, we
  need to show that each face of $\qbblackx n k$ with all vertices in
  $V_\IB$ is a face of $\IB$. To see this, note that each of the steps
  (B1)--\ref{i:last} maintains this property with respect to the
  interior boundary of the complex constructed thus far, and it is
  trivially satisfied for the starting complex $\qk{n-1}k$.

  (iii) Assume first that $n=2k+2$. For $j\in[2k-1]$, let $I'_j$ be
  the analogue of the independent set $I_j$, but defined in
  $C_{2k-1}$ and of size $k-1$. Thus, $I'_j =
  \Setx{j,j+2,\dots,j+2k-4}$ with arithmetic performed in
  $[2k-1]$. The assertion follows from the following property of the
  mapping $f$, valid for any $j\in[2k+1]$:
  \begin{equation*}
    f(I_j) =
    \begin{cases}
      I'_2 & \text{if $j=1$,}\\
      I'_1 & \text{if $j=2$,}\\
      I'_{j-1} & \text{if $3\leq j\leq 2k$,}\\
      I'_1 & \text{if $j=2k+1$.}
    \end{cases}
  \end{equation*}

  Let us now assume that $n > 2k+2$. Consider the mapping $h$ from the
  vertex set of $\qk{n-1}k$ to the vertex set of $\IB$, defined as
  follows:
  \begin{align*}
    h(\black A) &= \black{\clone n A}
      \text{\quad for $A\in\V{n-1}k$},\\
    h(\white B) &= \white{\clone{n-1} B} \text{\quad for
      $B\in\V[0]{n-1}k$},\\
    h(\white B) &= \white B \text{\quad for $B\in\V[+]{n-1}k$.}
  \end{align*}
  We claim that $h$ is a homomorphism of $2$-coloured complexes from
  $\qk{n-1}k$ to $\IB$. We need to show that for a face $\sigma$ of
  $\qk{n-1}k$, its image $h[\sigma]$ is a face of $\IB$. Indeed,
  consider a black vertex $\black A$ of $\sigma$ such that $A$ comes
  first in $\lex$. If $A\in\V[+]{n-1}k$, then a clone $\black{\clone n
    A}$ is added in step~\ref{i:black-out}, and the interior boundary
  of the resulting complex contains a face $\sigma_1 =
  \sigma\sm\Setx{\black A}\cup\Setx{\black{\clone n A}}$. In case
  $A\in\V[0]{n-1}k$, the same is true after performing
  steps~\ref{i:black-equ} and \ref{i:black-contract}. Proceeding
  similarly for the other black vertices of $\sigma_1$, we eventually
  obtain a face $\sigma'$ of $\IB$ in which each vertex $\black B$ of
  $\sigma$ is replaced by the clone $\black{\clone n B}$.

  The procedure for the white vertices $\white B$ of $\sigma'$ is
  similar: we replace each such vertex with $B\in\V[0]{n-1}k$ by its
  clone $\white{\clone{n-1}B}$ in one execution of steps~\ref{i:white}
  and \ref{i:white-contract}. Care needs to be taken if $B$ is
  singular, in which case these two steps are not executed. On the
  other hand, properties~\ref{i:sing} and \ref{i:sing-same} of
  $\qk{n-1}k$ imply that the cores of all black vertices adjacent to
  $\white B$ are the same, and the core of any white vertex adjacent
  to $\white B$ equals $\core B$. Thus the property that needs to be
  verified is just the existence of a single edge in $\qk{n-1}k$, and
  it follows by considering the nonsingular white vertex
  $\white{\clone{2k+1}B}$ instead of $\white B$.

  Thus, $h$ is a homomorphism as claimed. In addition, the above
  argument shows that each face of $\IB$ is the image of a face of
  $\qk{n-1}k$.

  Consider the exterior boundary $\qk{n-1}k$ of $\qbblackx n k$. By
  steps~\ref{i:upper}--\ref{i:glue} below, $\qk{n-1}k$ is obtained
  from $\qbblack{n-1}k$ and $\qbwhite{n-1}k$ by glueing them along
  their common boundary $\qk{n-2}k$ (viewed as the equator of
  $\qk{n-1}k$). Let $X$ be the set of vertices of $\qbblack{n-1}k$; it
  follows from the construction of $\qbblack{n-1}k$ and
  Lemma~\ref{l:all} below that
  \begin{align*}
    X = \bigcup_{A\in\V{n-2}k}\Setx{\black A,\white A}&\cup
    \Set{\black A}{A\in\bigcup_{i\geq 1} \V[2i-1]{n-1}k}\\
    &\cup \Set{\white A}{A\in\bigcup_{i\geq 1} \V[2i]{n-1}k}.
  \end{align*}
  In fact, $\qbblack{n-1}k$ is the induced subgraph of $\qk{n-1}k$ on
  $X$.

  As described in steps~\ref{i:identify}--\ref{i:relabel} below, the
  complex $\qbblack{n-1}k$ has been constructed as the union of the
  complex $\qbblackx{n-1}k$ and a complex, say $\KK^+$, isomorphic to
  $\qbwhite{n-3}{k-1}$. The intersection of these two subcomplexes is
  the interior boundary $\KK^0$ of $\qbblackx{n-1}k$. By part (i) of
  the lemma and induction, $\KK^0$ is isomorphic to $\qk{n-4}{k-1}$,
  and it is the induced subcomplex of $\qk{n-1}k$ on vertex set
  \begin{equation*}
    X^0 = \bigcup_{A\in\V[+]{n-2}k}\Setx{\black{\clone{n-1}
        A},\white{A}}.
  \end{equation*}
  
  By Observation~\ref{obs:onion}, $\KK^+$ is obtained from
  $\qbblack{n-1}k$ by removing the set of vertices
  \begin{equation*}
    Y = \Set{\black A}{A\in\V[0]{n-1} k}\cup
    \Set{\white A}{A\in\V[0]{n-2}k}.
  \end{equation*}

  Using Corollary~\ref{cor:cores} and inspecting the definition of
  $h$, we find that the restriction of $h$ to the vertex set of
  $\KK^+$, namely $X\sm Y$, is one-to-one. Let $\LL^+$ be the image of
  $\KK^+$ under $h$, and define $\LL^0$ as the image of
  $\KK^0$. Furthermore, let $\KK^-$ be the antipodal copy of $\KK^+$,
  and let $\LL^-$ be the image of $\KK^-$ under $h$. Since $h$ is also
  one-to-one when restricted to the vertex set of $\KK^-$, $\LL^-$ is
  isomorphic to $\qbblack{n-3}{k-1}$.

  From the definition of $h$, it follows that a vertex $\black A$ or
  $\white A$ is mapped by $h$ to $\LL^0$ if and only if $A\in
  \V[0]{n-1}k \cup \V[1]{n-1}k$. Consequently, the intersection of
  $\LL^+$ and $\LL^-$ equals $\LL^0$. It also follows that $\KK^+$ and
  $\KK^-$ are mapped isomorphically to $\LL^+$ and $\LL^-$,
  respectively.

  We have expressed $\IB$ as the union of two complexes, one
  isomorphic to $\qbblack{n-3}{k-1}$ and the other one to
  $\qbwhite{n-3}{k-1}$, intersecting in a subcomplex isomorphic to
  $\qk{n-4}{k-1}$. In view of steps~(K1)--\ref{i:glue} below, this
  implies that $\IB$ is isomorphic to $\qk{n-3}{k-1}$ as claimed.
\end{proof}

We can now finish the construction of $\qbblack n k$ (see
Figure~\ref{fig:qb} for an illustration):
\begin{enumerate}[label=(B\arabic*),start=6]
\item\label{i:identify} We identify the interior boundary of
  $\qbblackx n k$ with $\qk{n-3}{k-1}$ via the isomorphism of
  Lemma~\ref{l:interior}(iii).
\item\label{i:recurse} Applying the recursion, we extend this
  embedding of $\qk{n-3}{k-1}$ to an embedding of $\qbwhite{n-2}{k-1}$
  (note the change of colour).
\item\label{i:fillin} We form the complex $\qbblack n k$ as the union
  of $\qbblackx n k$ (constructed above) and $\qbwhite{n-2}{k-1}$.
\item \label{i:relabel} We give an explicit rule to relabel the
  vertices of $\qbwhite{n-2}{k-1}$ with elements of $\V n k$ in such a
  way that the labelling of the boundary matches the original
  labelling in $\qbblackx n k$ and each element of $\V n k$ appears as
  the label of a vertex (either a unique non-boundary vertex, or two
  antipodal boundary vertices).
\end{enumerate}

\begin{observation}\label{obs:onion}
  Let $Y$ be the set of vertices not contained in the interior
  boundary $\IB$ of $\qbblackx n k$. Then the complex $\qbblack n k\sm
  Y$, obtained by removing all the vertices in $Y$, is isomorphic to
  $\qbwhite{n-2}{k-1}$.
\end{observation}

\begin{figure}
  \centering
  \fig2\hf\fig9\\
  \fig4
  \caption{The construction of $\qbblack83$. Top left:
    $\qbblackx83$. Top right: $\qbwhite62$. Bottom: Filling in
    $\qbblackx83$ using $\qbwhite62$ produces $\qbblack83$. The
    labelling of the vertices inside the disk is discussed in
    rule~\ref{i:relabel}.}
  \label{fig:qb}
\end{figure}

To relabel the vertices of $\qbwhite{n-2}{k-1}$ so as to accomplish
step~\ref{i:relabel}, we will use the mapping $g_n$ of
Section~\ref{sec:combin}; recall that for $A\in\V{n-2}{k-1}$,
\begin{equation*}
  g_n(A) =
  \begin{cases}
    (A + 1) \cup \Setx 1 & \text{if $n-2\in A$,}\\
    (A + 1) \cup \Setx n & \text{otherwise.}     
  \end{cases}  
\end{equation*}

We relabel each black vertex $\black A$ of $\qbwhite{n-2}{k-1}$ to
$\black{g_n(A)}$ (cf. Figure~\ref{fig:qb}). A white vertex $\white A$
is relabelled to
\begin{align*}
  \white{g_n(A)} & \text{\quad if $A\in\V[+]{n-2}{k-1}$,}\\
  \white{g_{n-1}(A)} & \text{\quad otherwise.}
\end{align*}
We need to check that any vertex at the interior boundary of
$\qbblackx n k$ is mapped to itself by $g_n\circ f$ ($g_{n-1}\circ f$,
respectively). These are the vertices in the set $V_\IB$ defined in
Lemma~\ref{l:interior}(i). Recall that
\begin{equation*}
  V_\IB = \bigcup_{A\in \V[+]{n-1}k} \Setx{\black{\clone n A}, \white
    A}.
\end{equation*}
It follows from Lemma~\ref{l:bijection} that for $A\in\V[+]{n-1}k$,
$g_{n-1}(f(A)) = A$ and $g_n(f(\clone n A)) = \clone n A$, proving the
requested property. Further properties of the labelling will be proved
in Lemmas~\ref{l:all} and~\ref{l:disj} below.

We finally construct $\qk n k$ as follows:
\begin{enumerate}[label=(K\arabic*)]
\item\label{i:upper} We embed a deformed copy of $\qbblack n k$ in
  $\RR^{n-2k+1}$, with its vertices placed in the closed upper
  hemisphere $H^+$ of $S^{n-2k}$, in such a way that the embedded
  complex is boundary-antisymmetric (thus, the boundary $\qk{n-1}k$ is
  necessarily embedded in the `equator' $S^{n-2k-1}$).
\item\label{i:lower} Projecting each vertex of $\qbblack n k$ to its
  antipode in $S^{n-2k}$ and inverting its colour, we obtain a copy of
  $\qbwhite n k$ in $H^-$ that matches the former copy at the
  boundary.
\item\label{i:glue} $\qk n k$ is the result of glueing the two
  triangulated hemispheres together along their boundaries.
\end{enumerate}

In several lemmas, we now verify the properties of $\qk n k$ required
by Theorem~\ref{t:emb}.

\begin{lemma}\label{l:all}
  Each element of $\V n k$ appears as (the label of) a vertex of $\qk
  n k$.
\end{lemma}
\begin{proof}
  The assertion is easy to check for $n=2k+1$. If $n > 2k+1$, we
  inductively assume that it is true for $n-1$. Thus, any set $A\in
  \V{n-1}k$ is the label of a vertex of $\qk{n-1}k\subseteq\qk n k$.

  By Lemma~\ref{l:level}, it is sufficient to consider a set $A\in
  \V[+] n k$. Let $B = f(A)$, where $B\in\V{n-2}{k-1}$. By the
  induction hypothesis, $B$ is the label of a vertex of
  $\qk{n-2}{k-1}$, and hence of the complex $\qbwhite{n-2}{k-1}$ used
  in the construction of $\qbblack n k$. We may assume that the vertex
  is $\black B$ (the argument for $\white B$ being symmetric). The
  vertex was labelled with $g_n(B)$ in $\qbblack n k$; by
  Lemma~\ref{l:bijection}(ii), $g_n(f(A)) = A$ when $A\in \V[+] n k$,
  so $A$ does appear as a vertex label in $\qbblack n k$ and $\qk n
  k$.
\end{proof}

\begin{lemma}\label{l:bi-ind}
  Any edge $\black A\white B$ of $\qbblackx n k$, where
  $A,B\in\V{n-1}k$, is an edge of $\qk{n-1}k$.
\end{lemma}
\begin{proof}
  Consider an edge $\black A\white B$ of $\qbblackx n k$ but not of
  $\qk{n-1}k$, where $A,B\in\V{n-1}k$. In the construction of
  $\qbblackx n k$, which starts from $\qk{n-1}k$, the edge $\black
  A\white B$ was not added in
  steps~\ref{i:black-out}--\ref{i:black-contract} as these steps
  consist in adding clones of black vertices and contracting edges
  joining these clones. By Lemma~\ref{l:bd}(i), after
  step~\ref{i:black-contract} is completed, $\black A$ is not
  contained in the interior boundary of the resulting complex
  $\LL$. Consequently, steps~\ref{i:white}--\ref{i:white-contract} do
  not influence the set of edges incident with $\black A$. Thus, there
  is no step where $\black A\white B$ can be added, which is a
  contradiction.
\end{proof}

\begin{lemma}\label{l:disj}
  For any $A,B\in\V n k$ such that $\black A$ and $\white B$ are
  adjacent in $\qk n k$, $A\cap B=\emptyset$.
\end{lemma}
\begin{proof}
  We proceed by induction on $n$. The claim is easy to verify for
  $n=2k+1$. Assume that this is not the case; in addition, we may
  assume that $k > 1$. Let $\black A\white B$ be an edge of $\qk n k$.

  Without loss of generality, $\black A\white B$ is an edge of
  $\qbblack n k$ but not of $\qk{n-1}k$. Suppose first that $\black
  A\white B$ is an edge of $\qbblackx n k$. By the fact that each
  white vertex of $\qbblackx n k$ is a vertex of $\qk{n-1}k$ and by
  Lemma~\ref{l:bi-ind}, we find that $\black A$ is not a vertex of
  $\qk{n-1}k$. Inspecting steps~(B1)--\ref{i:last} of the
  construction, we observe that there are two possibilities:
  \begin{itemize}
  \item there is a set $C\in\V{n-1}k$ such that $A=\clone n C$ and
    $\black C\white B$ is an edge of $\qk{n-1}k$, or
  \item there are sets $C\in\V{n-1}k$, $D\in\V[0]{n-1}k$ such that
    $A=\clone n C$, $B = \clone{n-1}D$ and $\black C\white D$ is an
    edge of $\qk{n-1}k$.
  \end{itemize}
  In the first case, $C\cap B = \emptyset$ by the induction hypothesis
  and $n\notin C$, so $A\cap B = \emptyset$. In the second case, we
  similarly have $C\cap D = \emptyset$ by the induction hypothesis;
  since $n-1\notin A$ and $n\notin B$, we conclude $A\cap
  B=\emptyset$.

  It remains to consider the case that the edge $\black A\white B$ is
  not an edge of $\qbblackx n k$. By the construction of $\qbblack n
  k$, $\black{f(A)}\white{f(B)}$ is an edge of
  $\qbwhite{n-2}{k-1}$. By the induction hypothesis, $f(A)\cap f(B) =
  \emptyset$. By Lemma~\ref{l:bijection}, since $A,B\in\V[+] n k$, $A
  = g_n(f(A))$ and $B=g_n(f(B))$. The definition of $g_n$ shows that
  $A\cap B=\emptyset$ if one of $f(A)$, $f(B)$ contains $n-2$. Suppose
  thus that $n-2\notin f(A)\cup f(B)$. Then the $(n-2)$-level of both
  $f(A)$ and $f(B)$ is even. Since the vertices $\black{f(A)}$ and
  $\white{f(B)}$ have different colours, the $(n-2)$-levels actually
  have to be zero by property~\ref{i:close} of $\qk{n-2}{k-1}$. Thus,
  $f(A),f(B)\in \V[0]{n-2}{k-1}$, so $\black{f(A)}\white{f(B)}$ is an
  edge of the exterior boundary $\qk{n-3}{k-1}$ of
  $\qbwhite{n-2}{k-1}$ --- but then $\black A\white B$ would be an
  edge of the exterior boundary of $\qbblackx n k$, a contradiction.
\end{proof}

Lemmas~\ref{l:all} and \ref{l:disj} imply part (iii) of
Theorem~\ref{t:emb}. Parts (i) and (iv) follow easily from the
construction. Part (ii) is a consequence of the following lemma:

\begin{lemma}\label{l:mono}
  The complex $\qk n k$ contains no monochromatic maximal faces.
\end{lemma}
\begin{proof}
  The claim is certainly true for $\qk{2k+1}k$. For $\qk{2k+2}k$, it
  follows from the fact that each maximal face of $\qbblackx{2k+2}k$
  contains a $1$-dimensional face of either the exterior boundary or
  the interior boundary, and is therefore not monochromatic. With
  these base cases in hand, the lemma is proved by induction on
  $n$. Assume that $\qk{n-1}k$ has no monochromatic maximal faces. The
  complex $\qbblackx n k$ is obtained by two operations: adding clones
  and contracting $1$-dimensional monochromatic faces. None of these
  operations can create a monochromatic maximal face, so $\qbblackx n
  k$ has no such faces. The rest follows using
  Lemma~\ref{l:interior}(ii) and induction.
\end{proof}

It only remains to check properties~\ref{i:close}--\ref{i:sing-same}
of $\qk n k$. This routine verification
 is left to the reader.

This concludes the proof of Theorem~\ref{t:emb}. We remark that the
explicit construction of $\qk{2k+1}k$ and $\qbblackx{2k+2}k$ can, with
minor modifications, be interpreted as a particular case of the
general construction. Since a uniform treatment would make the
exposition more complicated, we prefer to present these special cases
separately.

Theorem~\ref{t:main} is a direct consequence of Theorem~\ref{t:emb}
and the results in~\cite{KS15}. Indeed, let the graph $QG(n,k)$ be
obtained from the associated graph of $\qk n k$ by identifying
antipodal pairs of vertices (and discarding the colours).  By
Theorem~\ref{t:emb} (i)--(ii) and \cite[Lemma~3.2]{KS15}, $QG(n,k)$ is
a quadrangulation of the projective space
$\proj{n-2k}$. Theorem~\ref{t:emb} (iii) implies that the
quadrangulation is a spanning subgraph of $SG(n,k)$, while part (iv)
implies that $QG(n,k)$ contains the $(2k+1)$-circuit $QG(2k+1,k)$ and
is therefore non-bipartite. Finally, by \cite[Theorem~1.1]{KS15} and
the easy upper bound on $\chi(SG(n,k))$, the chromatic number of
$QG(n,k)$ is $n-2k+2$.


\section{Conclusion}
\label{sec:properties}

We conclude this paper with two open problems.

While the proof of Theorem~\ref{t:emb} provides a recursive
characterisation of the pairs of sets in $\V n k$ that are adjacent in
the graph $QG(n,k)$, it would be desirable to define this
graph directly, without recursion. We have no such
definition so far.

Recall that the Schrijver graph $SG(n,k)$ is a vertex-critical
subgraph of the Kneser graph $KG(n,k)$ with the same chromatic number,
namely $n-2k+2$. By Theorem~\ref{t:main}, the spanning subgraph
$QG(n,k)$ of $SG(n,k)$ has the same chromatic number, and we
conjecture that it is the natural next step in the direction set by
Schrijver:
\begin{conjecture}
  For any $k\geq 1$ and $n\geq 2k+1$, $QG(n,k)$ is edge-critical.
\end{conjecture}
The conjecture is clearly true for $n=2k+1$ (odd cycles) and its
validity for $n=2k+2$ can be derived from a result of Gimbel and
Thomassen~\cite{GT97}.

\section*{Acknowledgements}

This project was started while the second author was visiting Fabricio
Benevides and V\'ictor Campos at the Universidade Federal do Cear\'a.


\end{document}